\documentclass[11pt]{amsart}
\usepackage{amsmath,amssymb,latexsym,cite}
\usepackage[small]{caption}
\usepackage{graphicx,wasysym,overpic,tikz,color}
\usepackage{subfigure}

\usepackage[colorlinks=true,urlcolor=blue,
citecolor=red,linkcolor=blue,linktocpage,pdfpagelabels,bookmarksnumbered,bookmarksopen]{hyperref}
\usepackage[english]{babel}

\usepackage[left=2.6cm,right=2.6cm,top=2.75cm,bottom=2.75cm]{geometry}

\numberwithin{equation}{section}
\newtheorem{theorem}{Theorem}[section]
\newtheorem{proposition}[theorem]{Proposition}
\newtheorem{lemma}[theorem]{Lemma}
\newtheorem{remark}[theorem]{Remark}
\newtheorem{corollary}[theorem]{Corollary}
\newtheorem{definition}[theorem]{Definition}
\theoremstyle{definition}

\newcommand{\N}{{\mathbb N}}
\newcommand{\R}{{\mathbb R}}

\newcommand{\dvg}{{\rm div}}

\newcommand{\eps}{\varepsilon}

\newcommand{\calM}{{\mathcal M}}

\newcommand{\set}[1]{\left\{#1\right\}}
\newcommand{\closure}[1]{\overline{#1}}
\newcommand{\bdry}[1]{\partial #1}
\newcommand{\half}{\frac{1}{2}}
\newcommand{\bgset}[1]{\big\{#1\big\}}
\newcommand{\restr}[2]{\left.#1\right|_{#2}}
\newenvironment{enumroman}{\begin{enumerate}

}{\end{enumerate}}
\DeclareMathOperator{\codim}{codim}

\title[Symmetry results for elliptic equations]{On symmetry
results for elliptic \\ equations with convex nonlinearities}

\author[K.\ Perera]{Kanishka Perera}
\author[M.\ Squassina]{Marco Squassina}

\address{Department of Mathematical Sciences \newline\indent
Florida Institute of Technology \newline\indent
150 West University Boulevard,
Melbourne, FL 32901-6975 USA}
\email{kperera@fit.edu}

\address{Dipartimento di Informatica
\newline\indent
Universit\`a degli Studi di Verona
\newline\indent
C\'a Vignal 2, Strada Le Grazie 15, I-37134 Verona, Italy}
\email{marco.squassina@univr.it}

\thanks{The second author was supported by 2009 national MIUR project: {\em ``Variational and Topological
Methods in the Study of Nonlinear Phenomena''}}


\begin{document}

\begin{abstract}
We investigate partial symmetry of solutions to semi-linear and quasi-linear elliptic problems with convex nonlinearities, in domains that are either axially symmetric or radially symmetric.
\end{abstract}

\maketitle

\section{Introduction}

\noindent
Let $\Omega$ be a smooth bounded domain in $\R^N$, $N\geq 2$. The goals of this paper are twofold. On the one hand, we extend some symmetry results in axially symmetric domains developed in \cite{pacella,pacellaweth} for the semi-linear elliptic equation with a convex nonlinearity
\begin{equation}
\label{prob-semi}
\begin{cases}
-\Delta u=f(x,u)   & \text{in $\Omega$,} \\
\noalign{\vskip2pt}
\,u=0   & \text{on $\partial\Omega$,}
\end{cases}
\end{equation}
to a framework where the energy functional naturally associated with \eqref{prob-semi} is of class $C^1$ but not of class $C^2$, that is to say when $f$ is continuous but not differentiable in the second argument. We give sufficient conditions for symmetry in terms of the local minimality of zero for certain related functionals (see the precise statements in Proposition~\ref{Proposition 1.1} and Corollary~\ref{Corollary 1.6}). In addition, we shall provide a further application to constrained minimization problems with convex nonlinearities in Theorem \ref{applConstrained}. In the framework of Morse theory, problems with the same level of regularity were investigated in \cite{bartschdegio} exploiting suitable tools of nonsmooth analysis. As pointed out in \cite{bartschdegio}, the extension to the nondifferentiable case is worthwhile for certain problems in mathematical ecology where one has to deal with jumping type nonlinearities. It is well-known that, under stronger assumptions on $\Omega$ and a monotonicity condition on the mapping $|x|\mapsto f(|x|,s)$, symmetry results can be achieved by the celebrated moving plane method (see, e.g., \cite{serr,gnn}). Other partial symmetry results in the framework of symmetrization and polarization theory were obtained in \cite{bwwi,smewil}.

On the other hand, assuming now that $f$ is smooth enough and it grows at infinity sufficiently fast, we obtain some symmetry results for the quasi-linear elliptic problem
\begin{equation}
\label{prob}
\begin{cases}
-\dvg(a(u)Du)+\frac{a'(u)}{2}|Du|^2=f(x,u)   & \text{in $\Omega$,} \\
\noalign{\vskip2pt}
\,u=0   & \text{on $\partial\Omega$,}
\end{cases}
\end{equation}
where $a:\R\to\R$ is smooth, positive and bounded away from zero. To this aim, we use a suitable change of variable procedure, namely, we transform the quasi-linear problem into an associated semi-linear problem $-\Delta v=h(x,v)$, whose nonlinearity $h$ depends both on $a$ and $f$. By investigating the convexity or strict convexity properties of the mapping $s\mapsto h(x,s)$, we can then apply the symmetry results obtained in \cite{pacella,pacellaweth} for the semi-linear case, and finally return to symmetry properties for the original problem (see Theorems~\ref{primapro0}, \ref{primapro-000}, \ref{primapro} and \ref{nodal-quasi} for the precise statements). A similar method has been employed in a recent paper of the second author jointly with F.\ Gladiali \cite{glasqu}, that deals with boundary blow-up solutions. These kinds of quasi-linear problems have been studied since 1995 in the framework of non-smooth critical point theory, being formally associated with (merely) continuous or lower semi-continuous functionals $J:H^1_0(\Omega)\to\R\cup\{+\infty\}$. Some recent applications involving \eqref{prob} have arisen in the study of the so called quasi-linear Schr\"odinger equation (see \cite{CJS} and the references therein). Some other applications can be traced back to differential geometry on manifolds with a general metric depending upon the solution itself. We refer the interested reader to 
the monograph \cite{iobook} of the second author and to the references therein for further details.

\section{Symmetry for semi-linear problems} \label{semilinearsection}

\noindent
Let $\Omega$ be a bounded domain in $\R^N,\, N \ge 2$, that contains the origin and is symmetric with respect to the hyperplane
\[
T = \set{x = (x_1,\dots,x_N) \in \R^N : x_1 = 0},
\]
and let $u_0 \in C^2(\Omega) \cap C(\closure{\Omega})$ be a classical solution of the problem
\begin{equation} \label{1.1}
\left\{\begin{aligned}
- \Delta u & = f(x,u) && \text{in } \Omega\\[5pt]
u & = 0 && \text{on } \bdry{\Omega},
\end{aligned}\right.
\end{equation}
where $f$ is a Carath\'{e}odory function on $\Omega \times \R$ that is even in $x_1$. In this section we study the symmetry properties of $u_0$ with respect to $x_1$ when $f$ is convex in the second variable.

We assume that $f$ satisfies the growth condition
\begin{equation} \label{1.2}
|f(x,t)| \le C \left(|t|^{r-1} + 1\right) \quad \text{for a.a. } x \in \Omega \text{ and all } t \in \R,
\end{equation}
where $C > 0$, $r > 1$, and $r < 2N/(N-2)$ if $N \ge 3$. Then $u_0$ is a critical point of the $C^1$-functional
\begin{equation} \label{1.3}
\Phi(u) = \int_\Omega \half\, |\nabla u|^2 - F(x,u), \quad u \in H^1_0(\Omega),
\end{equation}
where $F(x,t) = \int_0^t f(x,s)\, ds$. So $u = 0$ is a critical point of
\begin{align*}
\Psi(u) =\; & \Phi(u + u_0) - \Phi(u_0)\\[7.5pt]
=\; & \int_\Omega \half\, |\nabla u|^2 + f(x,u_0)\, u - F(x,u + u_0) + F(x,u_0),
\end{align*}
where we have used the fact that $u_0$ solves \eqref{1.1} to write $\int_\Omega \nabla u_0 \cdot \nabla u = \int_\Omega f(x,u_0)\, u$. Set
\[
\Omega_\pm = \bgset{x \in \Omega : x_1 \gtrless 0}, \qquad \Psi_\pm = \restr{\Psi}{H^1_0(\Omega_\pm)},
\]
and note that $u = 0$ is also a critical point of $\Psi_\pm$. We will prove that $u_0$ is even in $x_1$ under assumptions that involve the convexity of $f$ in $t$ and the type of critical point that $\Psi_\pm$ or $\Psi$ has at $u = 0$.

\noindent
Let $\widetilde{x} := (- x_1,x_2,\dots,x_N)$ be the reflection of $x$ on $T$ and let
\[
u_\pm(x) := u_0(\widetilde{x}) - u_0(x), \quad x \in \Omega_\pm.
\]
Since $u_+(x) = - u_-(\widetilde{x})$, if $u_\pm \ge 0$, then $u_+ = 0$ and hence $u_0(\widetilde{x}) = u_0(x)$. Let $u_\pm^- = \max \set{- u_\pm,0}$ be the negative parts of $u_\pm$. We assume
\begin{enumerate}
\item[(C$_1$)] for a.a. $x \in \Omega_\pm$ such that $u_\pm^-(x) \ne 0$, $f(x,\cdot)$ is convex on $[u_0(\widetilde{x}),u_0(x)]$.
\end{enumerate}
In particular, if $u_0 \ge 0$ (resp. $\le 0$), it suffices to assume that $f(x,\cdot)$ is convex on $[0,\max u_0]$ (resp. $[\min u_0,0]$) for a.a. $x \in \Omega$.

\begin{proposition} \label{Proposition 1.1}
Assume \eqref{1.2} and {\em (C$_1$)}. If $u = 0$ is a strict local minimizer of $\Psi_\pm$, then $u_0$ is even in $x_1$. If we have strict convexity in {\em (C$_1$)}, then it suffices to assume that $u = 0$ is a local minimizer of $\Psi_\pm$.
\end{proposition}

\noindent
This proposition is immediate from the lemma below, which implies that $u_\pm^- = 0$.

\begin{lemma} \label{Lemma 1.2}
If \eqref{1.2} and {\em (C$_1$)} hold, then
\begin{multline*}
\frac{d}{dt}\, \Psi_\pm(- tu_\pm^-) = \int_{\Omega_\pm} \big[f(x,(1 - t)\, u_0(x) + tu_0(\widetilde{x}))\\[7.5pt]
- (1 - t)\, f(x,u_0(x)) - tf(x,u_0(\widetilde{x}))\big]\, u_\pm^-(x) \le 0 \quad \forall t \in [0,1].
\end{multline*}
If we have strict convexity in {\em (C$_1$)} and $u_\pm^- \ne 0$, then the strict inequality holds for $t \in (0,1)$.
\end{lemma}

\begin{proof}
Since $u_0$ solves \eqref{1.1}, $u_\pm$ solve
\begin{equation} \label{1.4}
\left\{\begin{aligned}
- \Delta u & = f(x,u + u_0) - f(x,u_0) && \text{in } \Omega_\pm\\[5pt]
u & = 0 && \text{on } \bdry{\Omega_\pm},
\end{aligned}\right.
\end{equation}
and testing with $u_\pm^-$ and using $u_\pm(x) + u_0(x) = u_0(\widetilde{x})$ gives
\[
\int_{\Omega_\pm} |\nabla u_\pm^-|^2 = \int_{\Omega_\pm} \big[f(x,u_0(x)) - f(x,u_0(\widetilde{x}))\big]\, u_\pm^-(x).
\]
Substitute into
\[
\frac{d}{dt}\, \Psi_\pm(- tu_\pm^-) = \int_{\Omega_\pm} t\, |\nabla u_\pm^-|^2 - f(x,u_0)\, u_\pm^- + f(x,u_0 - tu_\pm^-)\, u_\pm^-
\]
and note that $f(x,u_0 - tu_\pm^-)\, u_\pm^- = f(x,(1 - t)\, u_0(x) + tu_0(\widetilde{x}))\, u_\pm^-(x)$.
\end{proof}

\noindent
Now we assume that for each $M > 0$, there is a constant $C_M > 0$ such that
\begin{equation} \label{1.5}
|f(x,s) - f(x,t)| \le C_M\, |s - t| \quad \text{for a.a. } x \in \Omega \text{ and all } s, t \in [- M,M],
\end{equation}
and strengthen (C$_1$) to
\begin{enumerate}
\item[(C$_2$)] for a.a. $x \in \Omega_\pm$ such that $u_\pm^-(x) \ne 0$, $f(x,\cdot)$ is convex on $[u_0(\widetilde{x}),2u_0(x) - u_0(\widetilde{x})]$.
\end{enumerate}

\begin{proposition} \label{Proposition 1.3}
Assume \eqref{1.2}, \eqref{1.5}, {\em (C$_2$)}, and that $u_0$ has a critical point on $T \cap \Omega$. If $u_0$ is not even in $x_1$, then
\[
\Psi(su_+^- + tu_-^-) \le 0 \quad \forall (s,t) \in [- 1,1] \times [- 1,1].
\]
If we have strict convexity in {\em (C$_1$)}, then the strict inequality holds for $(s,t) \in (- 1,1) \times (- 1,1) \setminus \set{(0,0)}$.
\end{proposition}

\begin{lemma} \label{Lemma 1.4}
If \eqref{1.2} and \eqref{1.5} hold, then $u_0$ is even in $x_1$ in the following cases:
\begin{enumroman}
\item \label{2.2.i} $u_+ \ge 0$ in $\Omega_+$ and $\partial u_0/\partial x_1 \ge 0$ somewhere on $T \cap \Omega$,
\item \label{2.2.ii} $u_- \ge 0$ in $\Omega_-$ and $\partial u_0/\partial x_1 \le 0$ somewhere on $T \cap \Omega$.
\end{enumroman}
\end{lemma}

\begin{proof}
\ref{2.2.i} We will show that $u_+$ vanishes in $\Omega_+$. Suppose $u_+ > 0$ somewhere. Since $u_+$ solves \eqref{1.4}, then $u_+ > 0$ in $\Omega_+$ by the strong maximum principle and hence $\partial u_+/\partial x_1 > 0$ on $T \cap \Omega$ by the Hopf lemma (it is here that we use \eqref{1.5}). This is a contradiction since $\partial u_+/\partial x_1 = - 2\, \partial u_0/\partial x_1$ on $T \cap \Omega$. Proof in case \ref{2.2.ii} is similar.
\end{proof}

\begin{lemma} \label{Lemma 1.5}
If \eqref{1.2} and {\em (C$_2$)} hold, then $\Psi_\pm(tu_\pm^-) \le \Psi_\pm(- tu_\pm^-)$ for all $t \in [0,1]$.
\end{lemma}

\begin{proof}
We have
\begin{multline*}
\frac{d}{dt} \left[\Psi_\pm(tu_\pm^-) - \Psi_\pm(- tu_\pm^-)\right] = \int_{\Omega_\pm} \big[2f(x,u_0) - f(x,u_0 - tu_\pm^-)\\[7.5pt]
- f(x,u_0 + tu_\pm^-)\big]\, u_\pm^- \le 0 \quad \forall t \in [0,1]
\end{multline*}
since for a.a. $x \in \Omega_\pm$ such that $u_\pm^-(x) \ne 0$, $u_0(x) - tu_\pm^-(x) \in [u_0(\widetilde{x}),u_0(x)]$ and $u_0(x) + tu_\pm^-(x) \in [u_0(x),2u_0(x) - u_0(\widetilde{x})]$ for $t \in [0,1]$.
\end{proof}

\begin{proof}[Proof of Proposition \ref{Proposition 1.3}]
Since $\partial u_0/\partial x_1 = 0$ at a critical point of $u_0$ on $T \cap \Omega$, $u_\pm^- \ne 0$ by Lemma \ref{Lemma 1.4}. By Lemmas \ref{Lemma 1.5} and \ref{Lemma 1.2},
\begin{equation} \label{1.6}
\Psi_\pm(tu_\pm^-) \le \Psi_\pm(- tu_\pm^-) \le 0 \quad \forall t \in [0,1].
\end{equation}
Extending $u_\pm^-$ to functions in $H^1_0(\Omega)$ by setting them equal to zero outside $\Omega_\pm$, then
\begin{equation} \label{1.7}
\Psi(su_+^- + tu_-^-) = \Psi_+(su_+^-) + \Psi_-(tu_-^-) \le 0 \quad \forall (s,t) \in [- 1,1] \times [- 1,1]
\end{equation}
since $u_\pm^-$ have disjoint supports. If we have strict convexity in (C$_1$), then the second inequality in \eqref{1.6} is strict for $t \in (0,1)$ and hence the inequality in \eqref{1.7} is strict for $(s,t) \in (- 1,1) \times (- 1,1) \setminus \set{(0,0)}$.
\end{proof}

\noindent
We now specialize to the case where $\Omega$ is either a ball or an annulus centered at the origin $O$ of $\R^N$, and $f(\cdot,t)$ is radial for all $t \in \R$. If $u_0 \ne 0$, then it has a critical point at some $P \in \Omega$, and we may apply Proposition \ref{Proposition 1.3} to any hyperplane containing $O$ and $P$ to get the following

\begin{corollary} \label{Corollary 1.6}
Assume \eqref{1.2}, \eqref{1.5}, and that $f(|x|,\cdot)$ is convex for a.a. $x \in \Omega$. If $P \ne O$ and $u_0$ is not axially symmetric with respect to $OP$, or if $P = O$ and $u_0$ is not radially symmetric, then there is a $2$-dimensional subspace $V \subset H^1_0(\Omega)$ containing sign-definite functions such that $u_0$ is a local maximizer of $\restr{\Phi}{u_0 + V}$. If $f(|x|,\cdot)$ is strictly convex for a.a. $x \in \Omega$, then $u_0$ is a strict local maximizer of $\restr{\Phi}{u_0 + V}$. If $u_0 \ge 0$, the convexity assumptions are needed only on $[0,\infty)$.
\end{corollary}

\noindent
As an application of Corollary \ref{Corollary 1.6}, consider the problem of minimizing the functional $\Phi$ defined in \eqref{1.3} on the closed set
\[
\calM = \set{u \in H^1_0(\Omega) : \int_\Omega G(x,u) = 1},
\]
where $G(x,t) = \int_0^t g(x,s)\, ds$ for some Carath\'{e}odory function $g$ on $\Omega \times \R$ satisfying \eqref{1.2} and \eqref{1.5} with $g$ in place of $f$, such that $g(\cdot,t)$ is radial for all $t \in \R$. Let $u_0 \in C^2(\Omega) \cap C(\closure{\Omega})$ be a minimizer, and assume that $g(\cdot,u_0(\cdot)) \ne 0$. Then there is a neighborhood of $u_0$ in $\calM$ that is a $C^1$-submanifold of $H^1_0(\Omega)$ of codimension $1$, and $u_0$ solves
\[
\left\{\begin{aligned}
- \Delta u & = f(|x|,u) + \lambda\, g(|x|,u) && \text{in } \Omega\\[5pt]
u & = 0 && \text{on } \bdry{\Omega},
\end{aligned}\right.
\]
for some $\lambda \in \R$ by the Lagrange-multiplier rule.

\begin{theorem} \label{applConstrained}
Under the above hypotheses, assume that $f(|x|,\cdot) + \lambda\, g(|x|,\cdot)$ is strictly convex for a.a. $x \in \Omega$ and $g(\cdot,u_0(\cdot))$ is either positive a.e. or negative a.e.\ Then $u_0$ is axially symmetric. If $u_0 \ge 0$, the convexity assumption is needed only on $[0,\infty)$.
\end{theorem}

\begin{proof}
Suppose $u_0$ is not axially symmetric, and set
\[
\widetilde{\Phi}(u) = \int_\Omega \half\, |\nabla u|^2 - F(x,u) - \lambda\, G(x,u), \quad u \in H^1_0(\Omega).
\]
Then there is a $2$-dimensional subspace $V \subset H^1_0(\Omega)$ containing sign-definite functions such that $u_0$ is a strict local maximizer of $\widetilde{\Phi}|_{u_0 + V}$ by Corollary \ref{Corollary 1.6}. For $u \in \calM$,
\[
\widetilde{\Phi}(u) = \Phi(u) - \lambda \ge \Phi(u_0) - \lambda = \widetilde{\Phi}(u_0)
\]
since $u_0$ minimizes $\Phi|_{\calM}$. Thus, to obtain a contradiction, it suffices to show that every neighborhood of $u_0$ in $u_0 + V$ intersects $\calM$ at a point different from $u_0$. The tangent space to $\calM$ at $u_0$ consists of vectors $u$ such that
\[
\int_\Omega g(x,u_0)\, u = 0,
\]
which then have to change sign since $g(\cdot,u_0(\cdot))$ is either positive a.e.\ or negative a.e. Since $V$ contains sign-definite functions, it follows that $V$ is not tangent to $\calM$ at $u_0$. The desired conclusion then follows since $\dim V > \codim \calM$.
\end{proof}

\noindent
For example, consider the eigenvalue problem
\[
\left\{\begin{aligned}
- \Delta u & = \lambda\, g(|x|,u) && \text{in } \Omega\\[5pt]
u & = 0 && \text{on } \bdry{\Omega},
\end{aligned}\right.
\]
where $g(|x|,\cdot)$ is strictly convex for a.a. $x \in \Omega$. 
If $\int_\Omega g(x,u_0)\, u_0 \ge 0$, we have $\lambda \ge 0$ 
and then Theorem \ref{applConstrained} applies.
The existence of at least one minimizer with foliated Schwarz symmetry can by obtained (without any convexity requirements) by applying the symmetric constrained version of Ekeland's variational principle proved by the second author in \cite[Section 2.4]{london}.

\section{Symmetry for quasi-linear problems}

\noindent
In this section we shall consider the quasi-linear elliptic problem \eqref{prob}
described in the introduction. In order to give a precise characterization of the symmetry
of the solutions to \eqref{prob} in symmetric domains, we shall
convert the (quasi-linear) problem into a corresponding semi-linear
problem through a change of variable procedure involving the globally defined Cauchy problem
\begin{equation}
	\label{cauchy}
g'=\frac{1}{\sqrt{a\circ g}},\qquad g(0)=0.
\end{equation}
Assuming that $a$ is bounded away from zero from below, \eqref{cauchy}
admits a unique globally defined strictly increasing solution $g\in C^{m+1}(\R)$
provided that $a\in C^m(\R)$, for $m\in\N$. Furthermore, $g$ is odd whenever
$a$ is an even function. A simple direct computation shows that $u$ is a $C^2$
smooth solution to \eqref{prob} if and only if $v=g^{-1}(u)$ is a $C^2$ smooth
solution to the semi-linear problem
\begin{equation}
\label{prob-semil}
\begin{cases}
-\Delta v=h(x,v)   & \text{in $\Omega$,} \\
\noalign{\vskip2pt}
\,v=0   & \text{on $\partial\Omega$,}
\end{cases}
\end{equation}
where we have set $h(x,s):=f(x,g(s))a^{-1/2}(g(s))$ for $x\in\Omega$ and $s\in\R$. Formally,
problem~\eqref{prob} is associated with
the non-smooth functional $J$ defined by setting
\begin{equation}
	\label{defJ}
J(u):=\frac{1}{2}\int_\Omega a(u)|Du|^2-\int_\Omega F(x,u)
\end{equation}
while \eqref{prob-semil} is associated with the smoother functional
$I:H^1_0(\Omega)\to\R$ defined by
$$
I(v):=\frac{1}{2}\int_\Omega |Dv|^2-\int_\Omega K(x,v)
$$
where $K(x,s):=F(x,g(s))$ for all $x\in\Omega$ and $s\in\R$.
When $F(x,u)\in L^1(\Omega)$ for a given $u\in H^1_0(\Omega)$, then one can
associate to~\eqref{prob} the lower semi-continuous functional
$J:H^1_0(\Omega)\to\R\cup\{+\infty\}$
which operates as in \eqref{defJ} when $a(u)|Du|^2\in L^1(\Omega)$ while it is $+\infty$
in the opposite case. For $a$ bounded $J:H^1_0(\Omega)\to\R$ is continuous.
It can be shown \cite[Proposition 2.3]{glasqu0} that, assuming
\begin{equation}
\label{condit-growth}
\lim_{|s|\to +\infty}\frac{a(s)} {|s|^k}<\infty,\quad
\lim_{|s|\to +\infty}\frac{|f(x,s)|} {|s|^p}<\infty,\qquad k>1,\,\,\,
1<p<\frac{(k+1)N+2}{N-2}
\end{equation}
uniformly with respect to $x$, for every $\eps>0$, there exists $C_\eps>0$ such that
$$
|h(x,s)|\leq C_\eps+\eps|s|^{(N+2)/(N-2)},
\qquad\text{for all $x\in\Omega$ and all $s\in\R$,}
$$
which implies that $I\in C^1(H^1_0(\Omega))$. If, furthermore,
$s\mapsto h(x,s)$ is $C^1$, under 
\eqref{condit-growth} and similar one for $a'$ and $f'$, 
arguing as in \cite[Proposition 2.3]{glasqu0}, it follows that
for any $\eps>0$, there exists $C_\eps>0$ such that
$$
|h'(x,s)|\leq C_\eps+\eps|s|^{4/(N-2)},
\qquad\text{for all $x\in\Omega$ and all $s\in\R$.}
$$
In turn, for $v\in H^1_0(\Omega)$,
$$
I''(v)(\varphi,\psi)=\int_\Omega D\varphi\cdot D\psi-\int_\Omega h'(x,v)\varphi\psi,
\qquad \forall \varphi,\psi\in H^1_0(\Omega)
$$
is well defined. If, in addition, $u$ is a $C^2(\Omega)\cap C(\overline{\Omega})$
solution to \eqref{prob}, then
$(\varphi,\psi)\mapsto \int_\Omega D\varphi\cdot D\psi-\int_\Omega h'(x,v)\varphi\psi$ is well defined
without assuming growth conditions on $h'$, since $x\mapsto h'(x,v(x))$ is a continuous function on $\overline{\Omega}$,
$v$ being a solution to~\eqref{prob-semil} with $v\in C^2(\Omega)\cap C(\overline{\Omega})$.
\vskip3pt
\noindent
After the above connection between problems~\eqref{prob} and \eqref{prob-semil} is established,
of course one could provide some symmetry results in symmetric domains by using the results
that we have obtained in Section~\ref{semilinearsection}. On the other hand,
we prefer to add stronger regularity assumptions and provide more concrete statements,
by applying directly the results of~\cite{pacella,pacellaweth} by investigating the convexity
properties of the maps $s\mapsto h(x,s)$ and $s\mapsto h'(x,s)$.
\vskip3pt
\noindent
In order to state the main results of this section, the following definition is in order.

\begin{definition}\rm
	\label{morsedef}
For a (smooth) solution $u$ to the quasi-linear problem \eqref{prob}, we put
$$
m(u,J):=m(g^{-1}(u),I),
$$
and we say that $m(u,J)$ is the Morse index of $u$ with respect to $J.$
\end{definition}

\noindent
For a smooth solution $u$ to~\eqref{prob},
the number $m(g^{-1}(u),I)$ appearing in Definition~\ref{morsedef}
is defined, in a classical way, as the supremum of the dimensions of the
linear subspaces $V$ of $H^1_0(\Omega)$ such that the quadratic form
$\varphi\mapsto \int_\Omega |D\varphi|^2-\int_\Omega h'(x,v)\varphi^2$
is negative definite on $V$, where $v:=g^{-1}(u)$.
Defining $m(u,J)$ directly in a reasonable way seems difficult due to the
lack of regularity of $J$.
\vskip4pt
We are now ready to state our results. First, we have the following

\begin{theorem}
	\label{primapro0}
Let $\Omega$ be a domain in $\R^N$, $N\geq 2$, which contains the origin and is
symmetric with respect the hyperplane $\{x_1=0\}$ and convex in the $x_1$-direction.
Let $a(s)=1+|s|^k$ with $k>1$ and let $\psi:\R^N\to \R^+$ be continuous, even in the
$x_1$-variable and increasing in the $x_1$-variable in $\{x\in\Omega:x_1<0\}$.
Then, if $p>k+1$, any (smooth) solution $u$ to the problem
\begin{equation}
	\label{prosym00}
\begin{cases}
-\dvg(a(u)Du)+\frac{a'(u)}{2}|Du|^2=\psi(x)u^{p}   & \text{in $\Omega$,} \\
\noalign{\vskip2pt}
\,u>0   & \text{in $\Omega$,}  \\
\noalign{\vskip2pt}
\,u=0   & \text{on $\partial\Omega$,}
\end{cases}
\end{equation}
is symmetric with respect to $x_1$, that is $u(-x_1,x_2,\dots,x_N)=u(x_1,x_2,\dots,x_N)$.
\end{theorem}

\noindent
Secondly, we formulate the following result in radial domains.

\begin{theorem}
	\label{primapro-000}	
Let $\Omega$ be a ball or an annulus in $\R^N$, $N\geq 2$,
$a(s)=1+|s|^k$ with $k>1$, $p>k+1$ and $\psi:\R\to \R^+$
continuous. Consider a (smooth) index-one solution $u$ to the problem
\begin{equation}
	\label{prosym000}
\begin{cases}
-\dvg(a(u)Du)+\frac{a'(u)}{2}|Du|^2=\psi(|x|)u^{p}   & \text{in $\Omega$,} \\
\noalign{\vskip2pt}
\,u>0   & \text{in $\Omega$,}  \\
\noalign{\vskip2pt}
\,u=0   & \text{on $\partial\Omega$.}
\end{cases}
\end{equation}
Let $P\in\Omega$ be a maximum point of $u$ and denote by $r_p$ the axis passing through
the origin and $P$. Then the following facts hold:
\begin{enumerate}
	\item $u$ is axially symmetric with respect to $r_p$;
	\item if $\Omega$ is a ball and $P$ is the origin, then $u$ is radially symmetric;
	\item if $u$ is not radially symmetric, it is never symmetric with respect to any
	$(N-1)$-dimensional hyperplane passing through the origin and not
	passing through the axis $r_p$;
	\item if $u$ is not radially symmetric, all its critical points
	belong to the symmetry axis $r_p$.
\end{enumerate}
\end{theorem}

\noindent
In radial domains, we also have the following partial symmetry results. We recall that
a function $u$ is said to be foliated Schwarz symmetric if there exists a unit vector $\xi\in\R^N$
and a function $\eta:\R^+\times\R\to\R$ such that $u(x)=\eta(|x|,x\cdot\xi)$ and $\eta(r,\cdot)$
is nondecreasing, for all $r\geq 0$.

\begin{theorem}
	\label{primapro}
Let $\Omega$ be a ball or an annulus in $\R^N$, $N\geq 2$,
$a(s)=1+|s|^k$ with $k>1$ and let $\psi:\R\to \R^+$
be a continuous function. Then there exists $p_k>2$ such that
for every $p\geq p_k$, any (smooth) solution $u$ to
\begin{equation}
	\label{prosym1}
\begin{cases}
-\dvg(a(u)Du)+\frac{a'(u)}{2}|Du|^2=\psi(|x|)|u|^{p-1}u   & \text{in $\Omega$,} \\
\noalign{\vskip2pt}
\,u=0   & \text{on $\partial\Omega$,}
\end{cases}
\end{equation}
with Morse index $m(u,J)\leq N$ is foliated Schwarz symmetric. Furthermore,
if $\psi$ is constant, then the nodal set of any sign changing solution $u$ of
problem \eqref{prosym1} with Morse index $m(u,J)\leq N$ intersects the boundary $\partial\Omega$.
\end{theorem}

\begin{remark}\rm
Other types of nonlinearities $a$ and $f$ could be considered for which the
assertions of the previous theorems hold, such as exponential and logarithmic
type nonlinearities. The general idea is that the source $f$ should grow faster
than the quasi-linear diffusion $a$ as $s\to\infty$.
\end{remark}

\noindent
For a given $u:\Omega\to\R$, we denote by ${\rm nod}(u)\in\N\cup\{\infty\}$ the number
of connected components of $\Omega\setminus\{u^{-1}(0)\}$. We can now state the last result of this section.
It implies, in particular, that when the problem is autonomous,
index-one radial solutions have at most one nodal domain.

\begin{theorem}
	\label{nodal-quasi}
	Let $\Omega$ be a ball or an annulus in $\R^N$, $N\geq 2$,
	$a(s)=1+|s|^k$ with $k>1$, $\frac{k}{2}<p<\frac{(k+1)N+2}{N-2}$
	and let $u$ be any radial solution to
	\begin{equation}
		\label{autopb}
	\begin{cases}
	-\dvg(a(u)Du)+\frac{a'(u)}{2}|Du|^2=|u|^{p-1}u   & \text{in $\Omega$,} \\
	\noalign{\vskip2pt}
	\,u=0   & \text{on $\partial\Omega$.}
	\end{cases}
	\end{equation}
Then ${\rm nod}(u)\leq 1+\frac{m(u,J)}{N+1}$.
\end{theorem}

\subsection{Some convexity results}
\noindent
Assume now that, for each fixed $x\in\Omega$, the functions $s\mapsto a(s)$
and $s\mapsto f(x,s)$ are twice differentiable.
Observe that, by direct computation, we obtain
\begin{equation}
	\label{hprimo}
h'(x,s)=\frac{2f'(x,g(s))a(g(s))-f(x,g(s))a'(g(s))}{2a^2(g(s))},\qquad\text{for every $s\in\R$}.
\end{equation}
Furthermore, there holds
\begin{align}
	\label{hsecondo}
h''(x,s)=&\frac{1}{2}a^{-7/2}(g(s))\Big\{2f''(x,g(s))a^2(g(s))-3 f'(x,g(s))a'(g(s))a(g(s))  \notag \\
&-f(x,g(s))a''(g(s))a(g(s))+2f(x,g(s))(a'(g(s)))^2 \Big\},\qquad\text{for every $s\in\R$}.
\end{align}

\begin{proposition}
	\label{conv-h2}
Assume that $p>k+1$, $k\geq 2$, $\psi:\R^N\to \R^+$ is
a continuous function, and
\begin{equation}
		\label{expotpot}
f(x,s)=
\begin{cases}
\psi(x)s^p & \text{if $s\geq 0$},  \\
0 & \text{if $s<0$},
\end{cases}
\,\,\,\qquad
a(s)=1+|s|^k,\,\,\,\,\,\, s\in\R.
\end{equation}
Then the map $s\mapsto h(x,s)$ is convex on $\R$ and
strictly convex on $(0,+\infty)$ for all $x\in\Omega$.
\end{proposition}
\begin{proof}
Assume that $k\geq 2$ and $p>k+1$. In particular, the functions $a$ and $f(x,\cdot)$ are
of class $C^2$ on $\R$. Therefore, taking into account formula~\eqref{hsecondo}, we need to prove that
$$
2f''(x,s)a^2(s)-3 f'(x,s)a'(s)a(s) -f(x,s)a''(s)a(s)+2f(x,s)(a'(s))^2\geq 0,\qquad\text{for all $s\in\R$.}
$$
Hence, on account of~\eqref{expotpot}, this inequality is fulfilled on $\R^-$, and on $\R^+$ it reads as
$$
2p(p-1)s^{p-2}(1+s^k)^2-3pk s^{p+k-2}(1+s^k)-k(k-1)s^{p+k-2}(1+s^k)+2k^2 s^{p+2k-2}\geq 0.
$$
This can be rearranged as
$$
\Gamma_1s^{p+2k-2}+\Gamma_2s^{p+k-2}+\Gamma_3s^{p-2}\geq 0,\qquad\text{for all $s\geq 0$,}
$$
where we have set
$$
\Gamma_1=2p^2-(2+3k)p+k^2+k,\quad
\Gamma_2=4p^2-(4+3k)p-k(k-1),\quad
\Gamma_3=2p(p-1).
$$
Then, by assumption, $\Gamma_1=(p-k-1)(2p-k)>0$ and
$$
\Gamma_2=(p-k-1)\Big[4p+\frac{k(p-k+1)}{p-k-1}\Big]>0,
$$
concluding the proof.
\end{proof}

\begin{remark}\rm
	In the case $p<k+1$, in general the map $h$ fails to be convex. For instance, if $p=k=3$,
	Figure~\ref{fig:confronto} shows the plot of $h''$ becoming negative inside the range $[0,2]$.
	\begin{figure}[!ht]
	\includegraphics[width=3.2in]{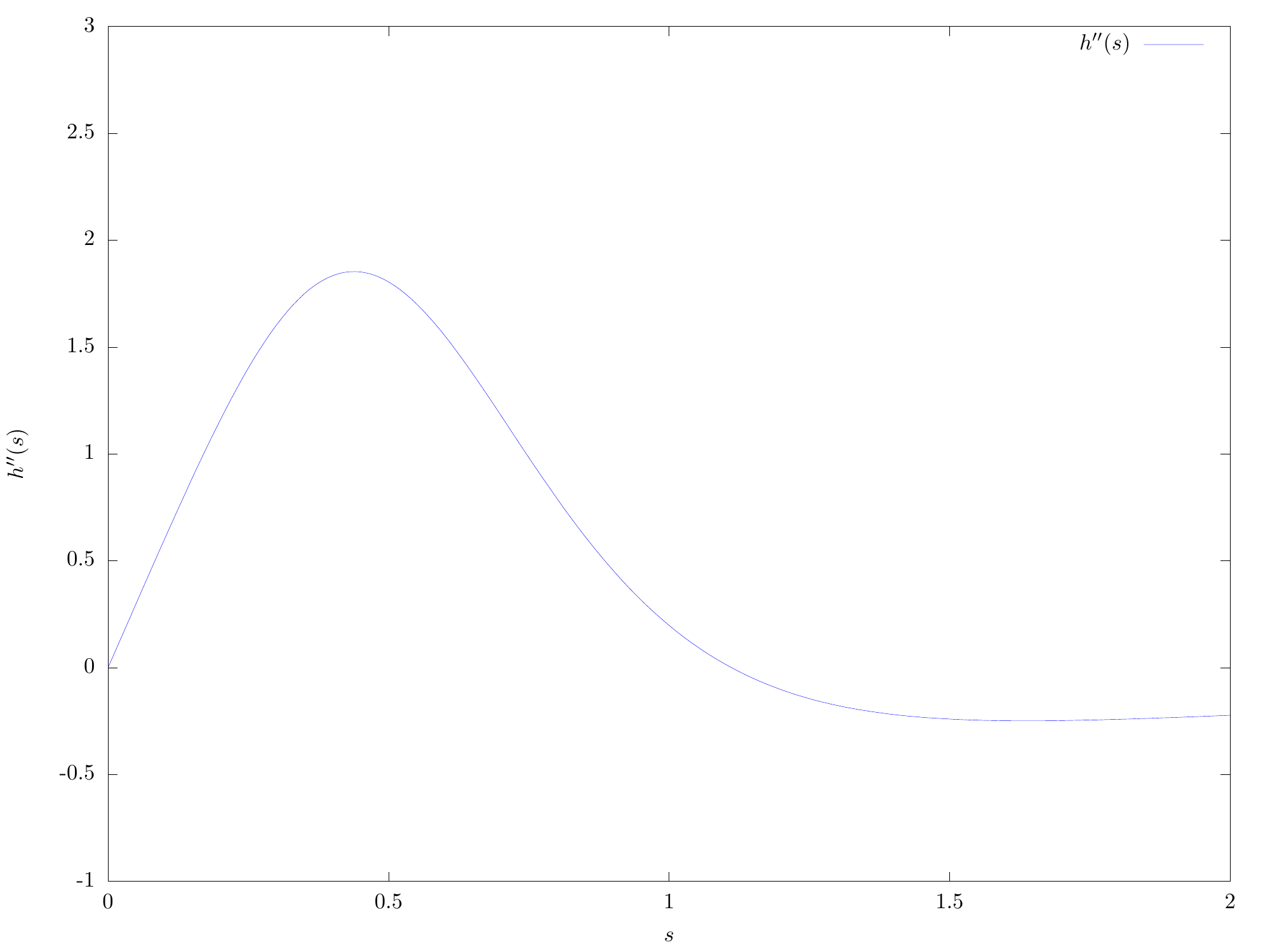}
	\caption{The figure shows that the second order derivative of $h$ becomes negative in the case $p=k=3$ and hence
	the convexity fails outside the range $p>k+1$.}
	\label{fig:confronto}
	\end{figure}
\end{remark}

\noindent
Observe now that, when $a$ and $f(x,\cdot)$ are of class $C^3$, from \eqref{hsecondo} we get
\begin{equation}
	\label{h3formul}
h'''(x,s)=\frac{1}{4}a^{-5}(g(s))\Big[2\Theta'(x,g(s)) a(g(s))-7a'(g(s))\Theta(x,g(s))\Big],
\end{equation}
where, for $x\in\Omega$ and $s\in\R$, we have set
\begin{align*}
\Theta(x,s)&:=2f''(x,s)a^2(s)-3 f'(x,s)a'(s)a(s)   \\
&-f(x,s)a''(s)a(s)+2f(x,s)(a'(s))^2,
\end{align*}
and, after some computations,
\begin{align*}
\Theta'(x,s)&=2f'''(x,s)a^2(s)
 +4a(s)a'(s)f''(x,s) \\
&-3f''(x,s)a'(s)a(s) -3f'(x,s)a''(s)a(s) \\
&-3f'(x,s)(a'(s))^2-f'(x,s)a''(s)a(s)  \\
&-f(x,s)a'''(s)a(s)-f(x,s)a''(s)a'(s)  \\
&+2f'(x,s)(a'(s))^2+4f(x,s)a'(s)a''(s).
\end{align*}

\noindent
Finally we can state the following convexity criterion for $h'$.

\begin{proposition}
	\label{h1conv-1}
Let $f$ and $a$ be as in \eqref{expotpot} with $k>1$.
Then the map $s\mapsto h'(x,s)$ is strictly convex on $(0,+\infty)$ for
all $p>2$ sufficiently large,
depending upon the value of $k$.
\end{proposition}
\begin{proof}
Notice first that the functions $a(\cdot)$ and $f(x,\cdot)$ can be
differentiated three times (and more) on the positive real line $(0,+\infty)$.
On account of formula \eqref{h3formul}, we need to prove that
$$
2\Theta'(x,s) a(s)-7a'(s)\Theta(x,s)> 0,\qquad\text{for all $s>0$.}
$$
Observe now that this means
\begin{align*}
& 4f'''(x,s)a^{3}(s)
 +8a^{2}(s)a'(s)f''(x,s)
-6f''(x,s)a'(s)a^{2}(s) -6f'(x,s)a''(s)a^{2}(s) \\
&-6f'(x,s)(a'(s))^2a(s)-2f'(x,s)a''(s)a^{2}(s)
-2f(x,s)a'''(s)a^{2}(s)-2f(x,s)a''(s)a'(s)a(s)  \\
&+4f'(x,s)a(s)(a'(s))^2+8f(x,s)a'(s)a''(s)a(s)
 -14f''(x,s)a^2(s)a'(s)+21f'(x,s)(a'(s))^2a(s)   \\
&+7f(x,s)a''(s)a'(s)a(s)-14f(x,s)(a'(s))^3>0,\qquad\text{for all $s>0$,}
\end{align*}
equivalently, since $\psi(x)\geq 0$,
\begin{align*}
& 4p(p-1)(p-2)s^{p-3}a^{3}(s)
 +8p(p-1)ks^{p+k-3}a^{2}(s)
-6p(p-1)ks^{p+k-3}a^{2}(s) \\
&-6pk(k-1)s^{p+k-3}a^{2}(s)
-6pk^2s^{p+2k-3}a(s)-2pk(k-1)s^{p+k-3}a^{2}(s) \\
&-2k(k-1)(k-2)s^{p+k-3}a^{2}(s)-2k^2(k-1)s^{p+2k-3}a(s)  \\
&+4pk^2s^{p+2k-3}a(s)+8k^2(k-1)s^{p+2k-3}a(s)
 -14kp(p-1)s^{p+k-3}a^2(s)+21pk^2s^{p+2k-3}a(s)   \\
&+7k^2(k-1)s^{p+2k-3}a(s)-14k^3s^{p+3k-3}>0,\qquad\text{for all $s>0$,}
\end{align*}
that is
\begin{align*}
& 4p(p-1)(p-2)a^{3}(s)
 +8p(p-1)ks^{k}a^{2}(s)
-6p(p-1)ks^{k}a^{2}(s) \\
&-6pk(k-1)s^{k}a^{2}(s)
-6pk^2s^{2k}a(s)-2pk(k-1)s^{k}a^{2}(s) \\
&-2k(k-1)(k-2)s^{k}a^{2}(s)-2k^2(k-1)s^{2k}a(s)  \\
&+4pk^2s^{2k}a(s)+8k^2(k-1)s^{2k}a(s)
 -14kp(p-1)s^ka^2(s)+21pk^2s^{2k}a(s)   \\
&+7k^2(k-1)s^{2k}a(s)-14k^3s^{3k}>0,\qquad\text{for all $s>0$,}
\end{align*}
that is
\begin{equation*}
\Pi_1(p) a^{3}(s)+\Pi_2(p)s^{k}a^{2}(s)+  \Pi_3(p)s^{2k}a(s)+\Pi_4(p) s^{3k}>0,\qquad\text{for $s>0$,}
\end{equation*}
where
\begin{align*}
&\Pi_1(p):=4p(p-1)(p-2), \\	
&\Pi_2(p):=-12kp(p-1)-8pk(k-1)-2k(k-1)(k-2),\\
&\Pi_3(p):=k^2(19p+13k-13), \\
&\Pi_4(p):=-14k^3.
\end{align*}
In turn this is fulfilled, for $k>1$, provided that $Q_p(s)>0$ for $s>0$, where
\begin{align*}
Q_p(s):=(\Pi_1(p)+\Pi_2(p)+\Pi_3(p)
&+\Pi_4(p))\cdot s^{3k}+(3\Pi_1(p)+2\Pi_2(p)+\Pi_3(p))\cdot s^{2k} \\
&+(3\Pi_1(p)+\Pi_2(p))\cdot s^k+\Pi_1(p).
\end{align*}
Taking into account that
$\Pi_1(p)={\mathcal O}(p^3)$ and $\Pi_j(p)=o(p^3)$ as $p\to\infty$
for all $j=2,\dots,6$, in turn there exists $p_k>2$ such that
for every $p\geq p_k$ it holds
\begin{equation}
	\label{Piposs}
\Pi_1(p)+\Pi_2(p)+\Pi_3(p)+\Pi_4(p)>0,\,\,\,
3\Pi_1(p)+2\Pi_2(p)+\Pi_3(p)>0,\,\,\,
3\Pi_1(p)+\Pi_2(p)>0.
\end{equation}
Then $Q_p(s)>0$ for all $s>0$, yielding the positivity of $h'''(s)$
for $s>0$ and, hence, the strict convexity of $h'$ on $[0,+\infty)$.
This concludes the proof.
\end{proof}

\noindent
Concerning the usual power nonlinearity $f(x,s)=\psi(x)|s|^{p-1}s$, we have the following

\begin{corollary}
	\label{h1conv-2}
Let $p>\max\{2,k+1\}$, $f(x,s)=\psi(x)|s|^{p-1}s$ for all $s\in\R$, where $\psi:\R^N\to \R^+$ is
continuous, and let $a(s)=1+|s|^k$ with $k>1$.
Then $s\mapsto h'(x,s)$ is strictly convex on $\R$
for every $x\in\Omega$ for any $p>2$ sufficiently large,
depending on the value of $k$.
\end{corollary}
\begin{proof}
	The assertion follows from Proposition~\ref{h1conv-1} after observing
	that, since $f$ is odd and $a$ is even (and hence $g$ is odd), the function $h$ is odd
	and, in turn, $h'$ is even with $h'(0)=0$. Hence $h'$ is strictly convex both on $(-\infty,0)$ and
	on $(0,+\infty)$ and hence on $\R$ since $h'$ is increasing
	on $(0,+\infty)$ as $p>k+1$ in light of Proposition~\ref{conv-h2}.
\end{proof}

\begin{remark}\rm
Explicit conditions on the magnitude of $p$ with respect to $k$
that guarantees the validity of the assertion of Proposition~\ref{h1conv-1}
can either obtained by solving directly the inequalities in \eqref{Piposs}
or searching for the absolute minimum point $s_\sharp>0$ of $Q_p$ on $(0,+\infty)$
which satisfy the quadratic equation for the unknown $\Xi=\Xi(p,k):=s_\sharp^k>0$
$$
3(\Pi_1(p)+\Pi_2(p)+\Pi_3(p)+\Pi_4(p))\Xi^2+(6\Pi_1(p)+4\Pi_2(p)+2\Pi_3(p))\Xi
+(3\Pi_1(p)+\Pi_2(p))=0,
$$
and finally imposing $Q_p(s_\sharp)>0$.
In the semi-linear (corresponing to the case where $a$ is a constant),
it follows that $h(x,s)=\psi(x)|s|^{p-1}s$
so that the requirement $p>2$ is necessary
for $s\mapsto h'(x,s)$ to be strictly convex on $\R$.
Figures~\ref{fig:confronto5} and \ref{fig:confronto6} show how
$h'''$ is pushed from negative to positive values provided that the value of $p$
is large enough in terms of $k$ ($k=2$ and $p=3.2,4,5,7$ respectively). For instance,
if the dimension $N$ is equal to $3$, the values of $p$ such that $h'''>0$
are below the threshold $3k+5=((k+1)N+2)/(N-2)$ appearing in \eqref{condit-growth}
for the growth of $f$ which makes the problem \eqref{prob-semil} subcritical and, thus, nice
for the existence theory via variational methods.
\end{remark}

	\begin{figure}
	\centering
	\mbox{\subfigure{\includegraphics[width=2.9in]{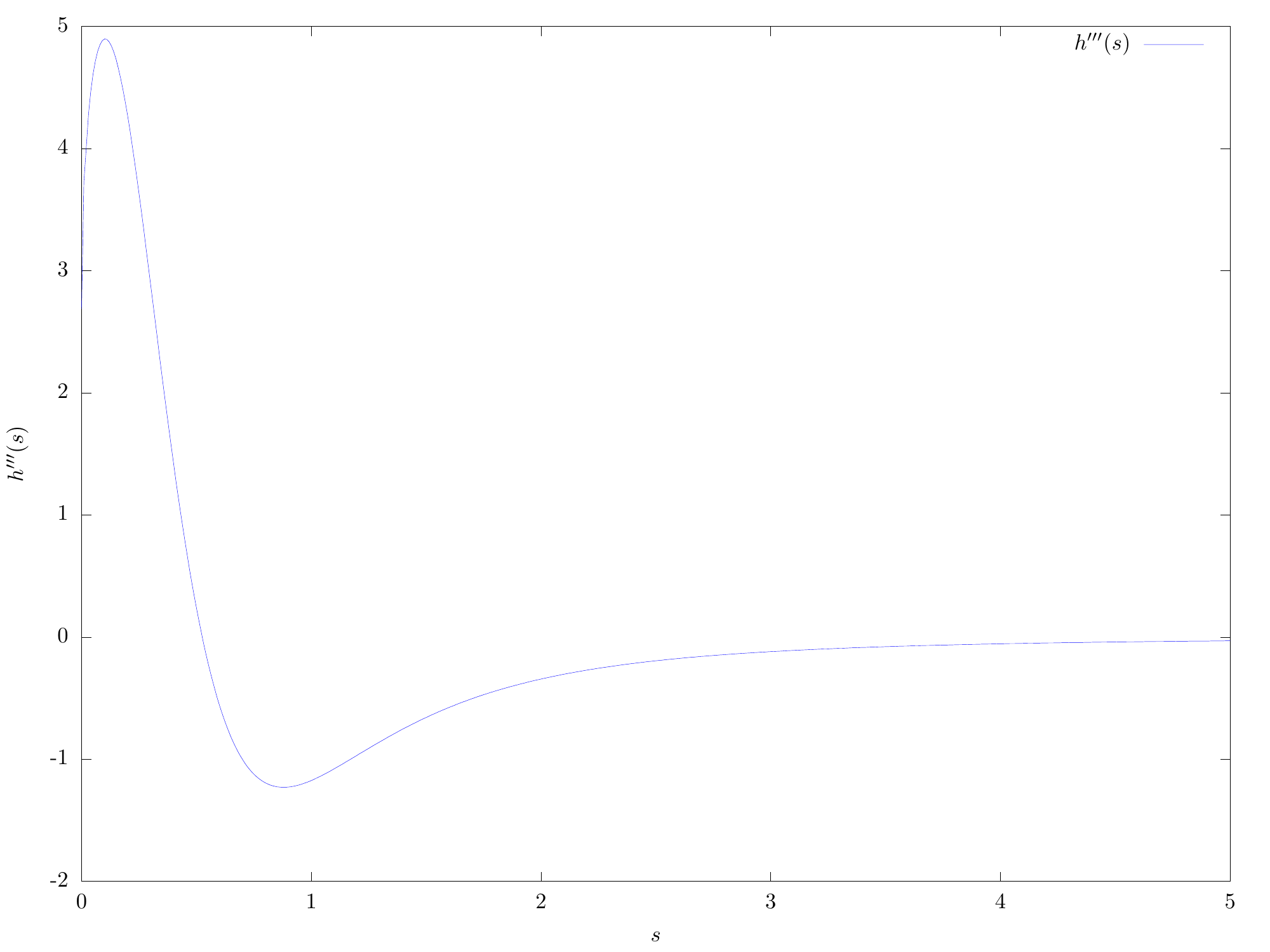}
	\quad
	\subfigure{\includegraphics[width=2.9in]{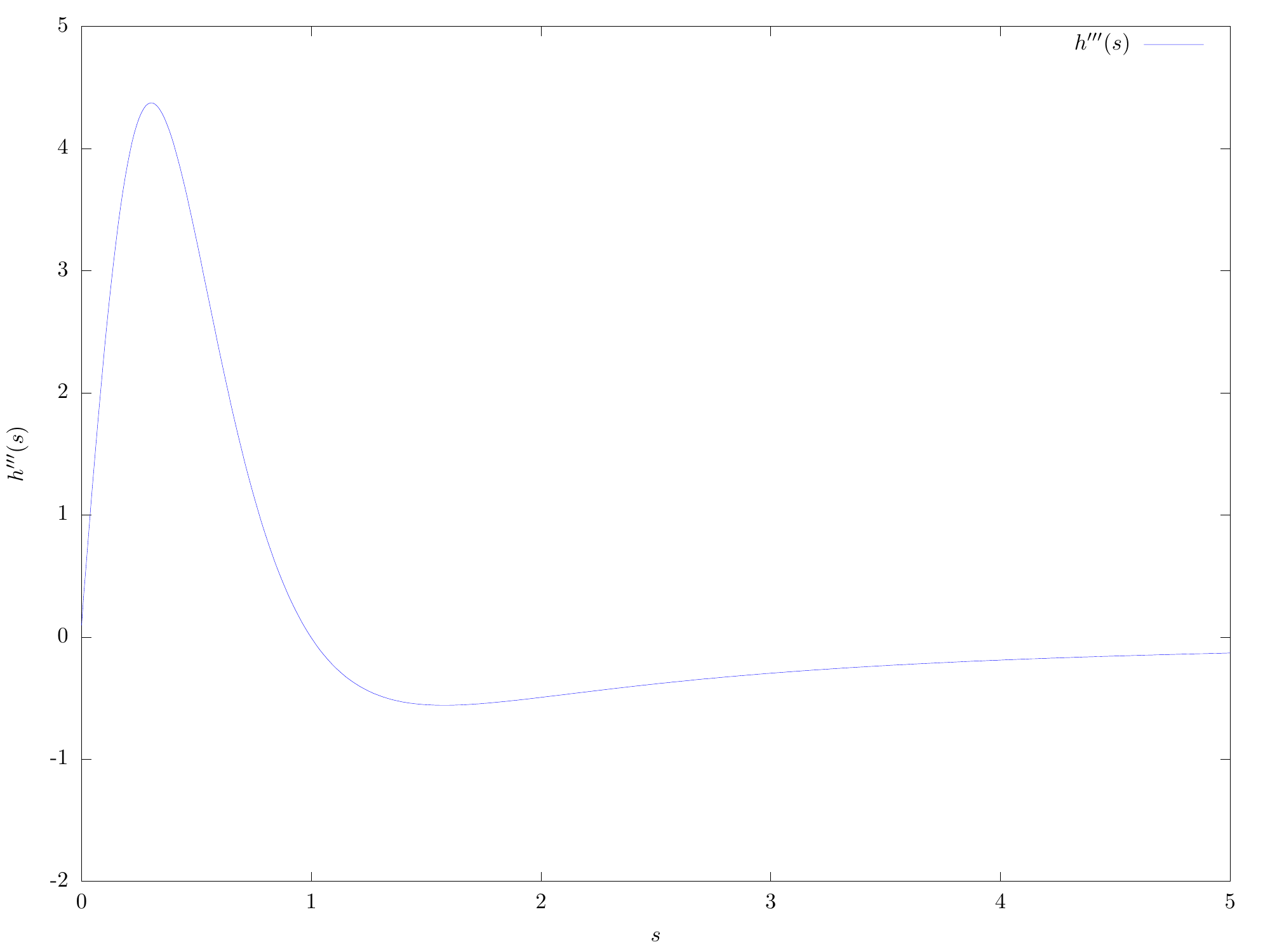} }}}
	\caption{The figure shows the plot of $h'''$ in the case $p=3.2$ (left) and $p=4$ (right) for $k=2$.}
	\label{fig:confronto5}
	\end{figure}

	\begin{figure}
	\centering
	\mbox{\subfigure{\includegraphics[width=2.9in]{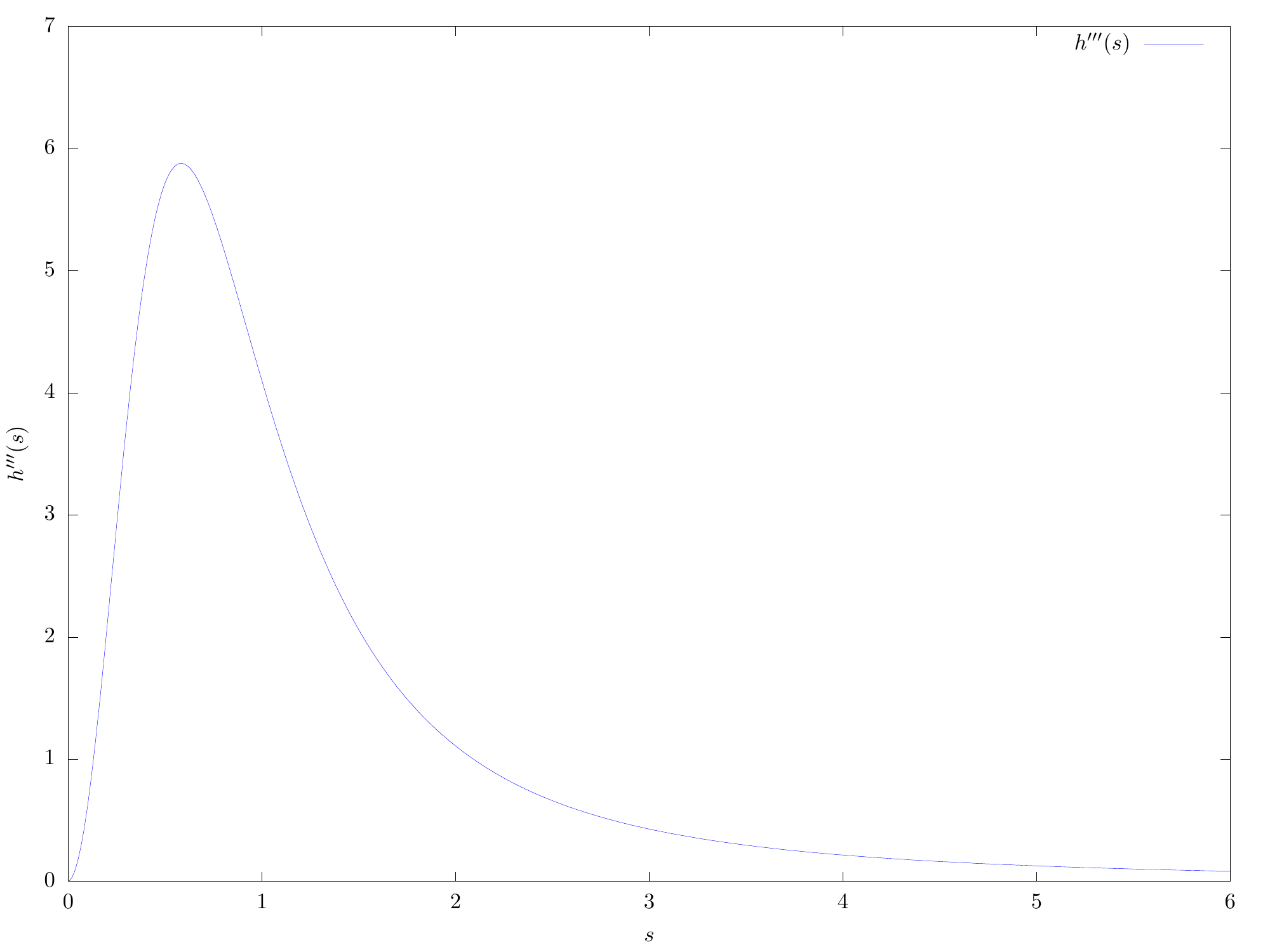}
	\quad
	\subfigure{\includegraphics[width=2.9in]{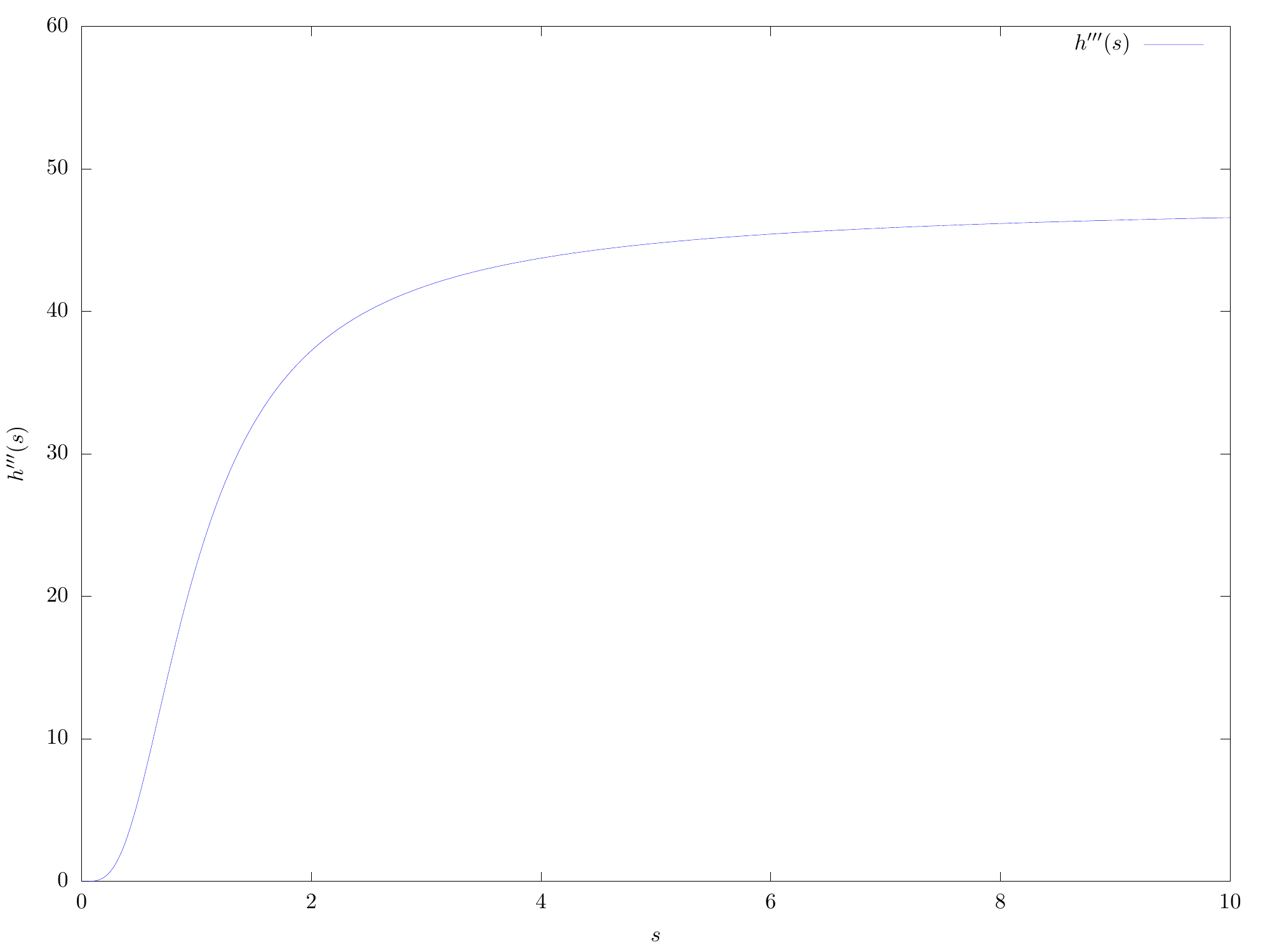} }}}
	\caption{The figure shows the plot of $h'''$ in the case $p=5$ (left) and $p=7$ (right) for $k=2$.}
	\label{fig:confronto6}
	\end{figure}

\subsection{Proofs of Theorems~\ref{primapro0}, \ref{primapro-000} and \ref{primapro}}
We are now ready to prove the previously stated symmetry results
for the quasi-linear problem.

\subsubsection{Proof of Theorem~\ref{primapro0}}

Given a (smooth) solution $u$ to \eqref{prosym00}, setting $v=g^{-1}(u)$, it
follows that $v$ is a smooth solution to $-\Delta v=h(x,v)$. Of course, for every $s\in\R$,
the function $h(\cdot,s)$ is continuous, even in the
$x_1$-variable and increasing in the $x_1$-variable in $\{x\in\Omega:x_1<0\}$. Now, in light
of Proposition~\ref{conv-h2}, it follows that the map $s\mapsto h(x,s)$ is
strictly convex on $(0,+\infty)$ for all $x\in\Omega$. Hence, by combining
\cite[Propositions 1.1 and 2.1]{pacella}, it follows that $v$ is symmetric with respect to $x_1$.
Therefore this yields
$$
u(-x_1,x_2,\dots,x_N)=g(v(-x_1,x_2,\dots,x_N))=g(v(x_1,x_2,\dots,x_N))=u(x_1,x_2,\dots,x_N),
$$
concluding the proof. \qed

\subsubsection{Proof of Theorem~\ref{primapro-000}}

	Let $u$ be a positive (smooth) index one solution to problem~\eqref{prosym00}. Hence, $v=g^{-1}(u)$ is a
	(smooth) solution to $-\Delta v=h(x,v)$. By virtue of Definition~\ref{morsedef}, it follows that $v$ has index $1$.
	Observe that $D_j u(x)=g'(v(x))D_j v(x)$ and
	$D^2_{ij}u(x)=g''(v(x))D_iv(x)D_jv(x)+g'(v(x))D^2_{ij} v(x)$ for all $x\in\Omega$ and any $i,j=1,\dots,N$.
	Since $g'>0$, $x_0$ is a critical point of $v$ if and only if $x_0$
	is a critical point of $u$, in which case
	${\mathcal H}_u(x_0)=g'(v(x_0)){\mathcal H}_v(x_0)$,
	where ${\mathcal H}_z(y)$ denotes the Hessian matrix of $z$ at $y$.
In fact, $P$ is a maximum point for $v$ also, since $v(\xi)=g^{-1}(u(\xi))\leq g^{-1}(u(P))=v(P)$
for all $\xi\in\Omega$, $g^{-1}$ being strictly increasing.
	On account of Proposition~\ref{conv-h2}, the proofs of assertions (1)-(3) follow as in
	the proof of Theorem~\ref{primapro0} by applying \cite[Theorem 3.1 (i), (ii)
	and (iii)]{pacella}. Concerning assertion (4), assume that $u$
	is not radially symmetric. Hence, $v=g^{-1}(u)$ is a
	nonradial (smooth) solution to $-\Delta v=h(x,v)$.
	Whence, by \cite[Theorem 3.1(4)]{pacella}, all its critical points
	belong to the symmetry axis $r_p$, that is to say $D_jv(\xi)=0$ implies $\xi\in r_P$.
	Since $D_j u(\xi)=g'(v(\xi))D_j v(\xi)$ for all $j$ and
	$g'>0$, $D_j u(\xi)=0$ implies $D_jv(\xi)=0$. Hence $\xi\in r_P$
	and the proof is complete. \qed

\subsubsection{Proof of Theorem~\ref{primapro}}

Let $u$ be any (smooth) solution to \eqref{prosym1} with Morse index $m(u,J)\leq N$.
Therefore, setting $v=g^{-1}(u)$, by Definition~\ref{morsedef}, $v$ is a smooth solution to $-\Delta v=h(|x|,v)$
with Morse index $m(v,I)\leq N$. In light of Corollary~\ref{h1conv-2}, the function $s\mapsto h(|x|,s)$
has a (strictly) convex derivative on $\R$ provided that $p$ is sufficiently large, depending on $k$. Then,
by virtue of \cite[Theorem 1.1]{pacellaweth}, it follows that $v$ is foliated Schwarz symmetric,
namely, there exists a unit vector $\xi\in\R^N$ such that $v(x)=\eta(|x|,\xi\cdot x)$ for some
function $\eta:\R^+\times\R\to\R$ such that $\eta(r,\cdot)$ is nondecreasing for any $r\geq 0$.
Then $u=(g\circ\eta)(|x|,\xi\cdot x)$ and $D_s (g\circ \eta)(|x|,s)=g'(\eta(|x|,s))D_s\eta(|x|,s)\geq 0$
since $g'>0$ on $\R$. This concludes the proof of the first assertion. The second assertion follows
by arguing analogously using~\cite[Theorem 1.2]{pacellaweth}. \qed

\subsubsection{Proof of Theorem~\ref{nodal-quasi}}
Let $u$ be any (smooth) radial solution to problem~\eqref{autopb}.
Then, setting $v=g^{-1}(u)$, it follows that $v$ is a (smooth) radial
solution to $-\Delta v=h(v)$. It is
readily seen that $h$ satisfies the requirement (f1)-(f4) (see the proof of \cite[Proposition 2.3]{glasqu}) needed
to apply \cite[Theorem 2.2]{bartschdegio}. In particular, for $p>k+1$, the map $s\mapsto h(s)/|s|$ is increasing
on $\R^-$ and on $\R^+$. Therefore by \cite[Theorem 2.2]{bartschdegio} it follows that
${\rm nod}(v)\leq 1+\frac{m(v,I)}{N+1}$. Recalling that $m(u,J)=m(v,I)$
and ${\rm nod}(u)={\rm nod}(v)$ (since $g$ vanishes only at $s=0$), the conclusion follows. \qed

\bigskip
\bigskip
\end{document}